\documentclass[a4paper,reqno]{amsart}

\usepackage{amsmath,amsfonts,amsthm,amssymb,hyperref,latexsym,mathrsfs,enumerate,pgfplots,tikz-cd}

\usetikzlibrary{decorations.markings,automata,arrows,positioning,calc}

\DeclareMathOperator{\conv}{\mathrm{conv}}
\DeclareMathOperator{\diag}{\mathrm{diag}}

\DeclareMathOperator{\coker}{\mathrm{coker}}
\DeclareMathOperator{\Homo}{\mathrm{Hom}}
\DeclareMathOperator{\sur}{\mathrm{Sur}}

\DeclareMathOperator{\cs}{\mathrm{C}^*}
\DeclareMathOperator{\mc}{\mathrm{MinCos}_q}
\DeclareMathOperator{\js}{\mathcal{Z}}
\DeclareMathOperator{\Aut}{\mathrm{Aut}}

\DeclareMathOperator{\aff}{\mathrm{Aff}}

\DeclareMathOperator{\rank}{\mathrm{rank}~}

\DeclareMathOperator{\Span}{\mathrm{span}}

\newcommand{\nn}{\mathbb{N}}
\newcommand{\zz}{\mathbb{Z}}

\newcommand{\rr}{\mathbb{R}}
\newcommand{\cc}{\mathbb{C}}
\newcommand{\pp}{\mathbb{P}}
\newcommand{\ee}{\mathbb{E}}
\newcommand{\g}{\mathbb{G}}

\makeatletter
\@namedef{subjclassname@2020}{%
  \textup{2020} Mathematics Subject Classification}
\makeatother

\newtheorem {theorem}{Theorem}[section]

\newtheorem {proposition}[theorem]{Proposition}

\theoremstyle {definition}

\newtheorem {example}[theorem]{Example}

\newtheorem {remark}[theorem]{Remark}

\numberwithin{equation}{section}

\title{Random amenable $\cs$-algebras}

\author[B.~Jacelon]{Bhishan Jacelon}
\address[B.~Jacelon]{
Institute of Mathematics of the Czech Academy of Sciences\\ \v{Z}itn\'{a} 25\\115 67 Prague 1\\Czech Republic}
\email{bjacelon@gmail.com}


\begin{document}

\begin{abstract}
What is the probability that a random UHF algebra is of infinite type? What is the probability that a random simple AI algebra has at most $k$ extremal traces? What is the expected value of the radius of comparison of a random Villadsen-type AH algebra? What is the probability that such an algebra is $\mathcal{Z}$-stable? What is the probability that a random Cuntz--Krieger algebra is purely infinite and simple, and what can be said about the distribution of its $K$-theory? By constructing $\cs$-algebras associated with suitable random (walks on) graphs, we provide context in which these are meaningful questions with computable answers.
\end{abstract}

\maketitle

\section{Introduction} \label{section:intro}

Inductive limit constructions lie at the very heart of operator algebras. Beginning with Murray and von Neumann's analysis of hyperfinite factors \cite{Murray:1943le}, continuing with Glimm's description of their $\cs$-analogues, the uniformly hyperfinite (UHF) algebras \cite{Glimm:1960qh}, and later with Elliott's pioneering work on approximately finite-dimensional (AF) algebras \cite{Elliott:1976kq}, and beyond, not only have inductive limits of type $\rm I$ algebras featured consistently and prominently in the natural $\cs$-world, but their classification has been a perennial pursuit of $\cs$-taxonomists.

In fact (see \cite[Theorem 13.50]{Gong:2020ud}), \emph{every} (unital) simple, separable $\cs$-algebra that has finite decomposition rank and satisfies the UCT (that is, every stably finite classifiable $\cs$-algebra) is isomorphic to an ASH algebra, that is, to an inductive limit of subhomogeneous $\cs$-algebras. Even on the other branch of the classifiable tree, purely infinite $\cs$-algebras can be built from limits of circle algebras by taking crossed products by trace-scaling endomorphisms (see \cite[Proposition 4.3.3]{Rordam:2002yu}). The known examples \cite{Villadsen:1998ys, Rordam:2003rz, Toms:2008vn, Toms:2008hl} of amenable $\cs$-algebras exhibiting Elliott invariants with more exotic behaviour also arise from inductive limits.

In this article, we address the question: what can be said about the distributions of the random $\cs$-algebras obtained by viewing these constructions as probabilistic rather than deterministic? To make sense of this, and to answer the questions posed in the abstract, we outline specific procedures for probabilistic model building and examine the distributions of suitable invariants. In particular, we will:
\begin{enumerate}[(i)]
\item as a warmup example, use random walks on the prime numbers to build UHF algebras of random type (see \S~\ref{section:uhf});
\item similarly build simple inductive limits of prime dimension drop algebras with random tracial state spaces, guided by random walks on the natural numbers (see \S~\ref{section:ai});
\item use random walks on binary trees to evolve Villadsen algebras whose radii of comparison are random variables with tractable distributions (see \S~\ref{section:ah});
\item investigate the distribution of the $K$-theory of graph $\cs$-algebras associated to random regular multigraphs (see \S~\ref{section:pi}).
\end{enumerate}

There are two sources of context for this line of investigation. First, our use of classifying invariants to develop various classes of $\cs$-algebras into probability spaces is telegraphed by the Borel parameterisations established in \cite{Farah:2013td,Farah:2014vt}. Second, in work that builds on the Erd\"{o}s-R\'{e}nyi construction of the random graph \cite{Erdos:1963vt}, it is shown in \cite{Droste:2003uh} how to probabilistically construct ($\omega$-categorical) universal homogeneous relational structures (briefly, one decides by coin toss whether or not a given tuple of generators satisfies a relation). What unites the universal UHF algebra $\mathcal{Q}$, the Jiang--Su algebra $\js$ and the stably projectionless algebra $\mathcal{W}$, apart from (strong) self-absorption, is that they are the Fra\"{i}ss\'{e} limits of suitable categories of tracial $\cs$-algebras (see \cite{Eagle:2016ww,Masumoto:2017wx,Jacelon:2021uc}). In other words, they are the generic objects that can be built as limits from their associated Fra\"{i}ss\'{e} classes. While for the most part we have found it more natural to work with the inductive limit structure inherent to $\mathcal{Q}$, $\js$, $\mathcal{W}$ and other $\cs$-algebras of interest, and consequently to appeal to the theory of Markov chains \S~\ref{section:mc}, the model employed in \S~\ref{section:pi} may be thought of as closer in spirit to the relational approach.

Computations of specific probabilities $\pp$ and expectations $\ee$ are included alongside our random constructions, but the reader should not attach much significance to, say, the assertion that the probability that a random simple inductive limit of prime dimension drop algebras has at most $k$ extremal traces is $\sum\limits_{i=0}^k \frac{k-i}{(k+1)2^{i+1}}$ (see Remark~\ref{remark:simplex}). Rather, the point being illustrated is that, with the right framework, such calculations are indeed possible. Our suggestion is that, whether via random walks or the creation of random pairings, (multi)graphs provide a natural and manageable way of probabilistically generating large classes of interesting $\cs$-algebras. And the take-home message should be that graph properties like recurrence or transience of random walks translate to almost-sure predictions about the structure of the associated random $\cs$-algebras (for example, a simple symmetric walk on the natural numbers almost surely generates the universal UHF algebra).

\subsection*{Acknowledgements} This research was supported by the GA\v{C}R project 22-07833K and RVO: 67985840. It germinated in conversations with Ali Asadi--Vasfi, Tristan Bice, Karen Strung and other members of the Prague NCG\&T group. I am grateful for the kind hospitality of my colleagues Dr.~Dane Mejias in Seattle and Dr.~Holly Morgan in Halifax during a visit to North America in the summer of 2022, when much of the article was completed.

\section{Markov chains} \label{section:mc}

The random constructions in \S~\ref{section:uhf} and \S~\ref{section:ai} are based on the following continuous-time Markov chain (specifically, a birth-and-death process). For an introduction to Markov chains, see \cite{Norris:1998wc}.

\subsection{The model} \label{subsection:mcmodel}

The random variable $X_0$ is specified by a probability distribution $\pi$ on $\nn=\{0,1,2,\dots\}$. At time $t\ge0$, the value of $X_t$ is governed by the rate matrix
\begin{equation} \label{eqn:rate}
Q = \begin{pmatrix} -(\lambda_0+\mu_0) & \lambda_0 &  & & \\ \mu_1 & -(\lambda_1+\mu_1) & \lambda_1 &  &  & \\  & \mu_2 & -(\lambda_2+\mu_2) & \lambda_2 & & \\ & & \ddots & \ddots & \ddots \end{pmatrix},
\end{equation}
where $\mu_i >0$ for $i\ge1$ and $\lambda_i >0$ for $i\ge0$. It was shown in \cite[Theorems 14,15]{Karlin:1957uz} (see also the introduction of \cite{Karlin:1957tf}) that, if the sequences
\[
\alpha_n:=\prod\limits_{i=1}^{n}\frac{\lambda_{i-1}}{\mu_i} \quad,\quad \beta_n:=\prod\limits_{i=1}^{n}\frac{\mu_i}{\lambda_i} \quad,\quad n\ge1
\]
are such that $\sum\limits_{n=1}^{\infty}(\alpha_n+\beta_n)=\infty$, then there exists a \emph{unique} matrix $P(t)$ associated to $Q$ that satisfies $P'(t)=QP(t)=P(t)Q$, $P(0)=I$, $P(s+t)=P(s)P(t)$, $P_{ij}(t)\ge0$ and $\sum_{j}P_{ij}(t)\le1$ (for all $s,t,i,j$ as appropriate), and that consequently determines the transition probabilities:
\[
\pp(X_{s+t}=j \mid X_{s}=i) = P_{ij}(t).
\]
While the continuous process provides important context, what we actually pay attention to is its jump chain, that is, the discrete Markov chain $(Y_n)_{n\in\nn}$ on $\nn$ whose initial distribution is $\pi$ and whose transition matrix $\Pi$ has entries
\begin{equation} \label{eqn:tm}
 \Pi_{ij} =
 \begin{cases}
 p_i:=\frac{\lambda_i}{\lambda_i+\mu_i} & \text{ if }\: j=i+1\\
 q_i:=\frac{\mu_i}{\lambda_i+\mu_i} & \text{ if }\: j=i-1\\
 0 & \text{ if }\: |i-j|>1.
 \end{cases}
\end{equation}
The jump chain is represented by the diagram
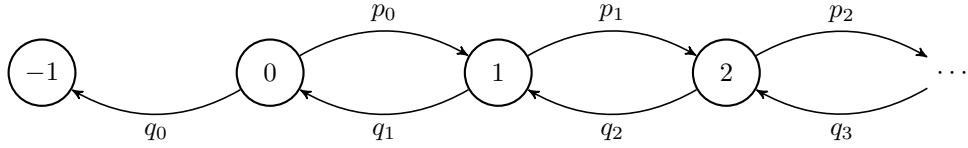
\[
	\begin{tikzpicture}[->, >=stealth', auto, semithick, node distance=3cm]
	
	\node[state][draw=black,thick,text=black,scale=1]    (A)                     {$-1$};
	\node[state][draw=black,thick,text=black,scale=1]    (B)[right of=A]   {$0$};
	\node[state][draw=black,thick,text=black,scale=1]    (C)[right of=B]   {$1$};
	\node[state][draw=black,thick,text=black,scale=1]    (D)[right of=C]   {$2$};
	\node   (E)[right of=D]   {$\cdots$};
	\path
	(B) edge[bend left,below]	node{$q_0$}	(A)
	edge[bend left,above]		node{$p_0$}	(C)
	(C) edge[bend left,below]	node{$q_1$}	(B)
	edge[bend left,above]		node{$p_1$}	(D)
	(D) edge[bend left,below]	node{$q_2$}	(C)
	edge[bend left,above]		node{$p_2$}       (E)
	(E) edge[bend left,below]	node{$q_3$}	(D);
	\end{tikzpicture}
\]
There are two possibilities for what happens at $i=0$:
\begin{enumerate}[(I)]
\item \label{item:forever} $0$ is a reflecting barrier (that is, $q_0=\mu_0=0$);
\item \label{item:retire} $-1$ is an absorbing state (that is, with probability $q_0=\frac{\mu_0}{\lambda_0+\mu_0}>0$, $Y_n$ moves from $0$ to $-1$ and stays there).
\end{enumerate}

We will use the process $(Y_n)_{n\in\nn}$ to construct a random UHF algebra in \S~\ref{section:uhf} and a random simple inductive limit of prime dimension drop algebras in \S~\ref{section:ai}. If we want a nonzero probability of ending up with a nontrivial finite structure (a finite matrix algebra or finite-dimensional tracial simplex), then we go with option (\ref{item:retire}).

\subsection{Absorption, recurrence and transience}

The behaviour of the Markov chain $(Y_n)_{n\in\nn}$ is determined by the growth of the sequences
\begin{equation} \label{eqn:growth}
a_n:=\prod\limits_{i=1}^{n}\frac{p_{i-1}}{q_i} \quad,\quad b_n:=\prod\limits_{i=1}^{n}\frac{q_i}{p_i} \quad,\quad n\ge1.
\end{equation}
Recall that every state is of one of three types:
\begin{enumerate}[(a)]
\item \emph{positive recurrent} (or \emph{ergodic}) if with probability $1$ it is visited infinitely often, the mean return time between visits being finite;
\item \emph{null recurrent} if with probability $1$ it is visited infinitely often, but with infinite mean return time between visits;
\item \emph{transient} if with probability $1$ it is visited only finitely many times.
\end{enumerate}
In Case (\ref{item:forever}), the chain is irreducible (that is, there is a nonzero probability of transitioning between any two given states), so all states are of the same type.

The following results are by now well known. See \cite[\S2]{Karlin:1959tm} for a discussion of how they can be deduced from the analysis of the continuous birth-and-death process carried out in \cite{Karlin:1957tf}.

\begin{theorem} \label{thm:forever}
In Case (\ref{item:forever}), the chain $(Y_n)_{n\in\nn}$ is:
\begin{enumerate}[(a)]
\item positive recurrent if and only if $\sum\limits_{n=1}^\infty b_n = \infty$ and $\sum\limits_{n=1}^\infty a_n < \infty$;
\item null recurrent if and only if $\sum\limits_{n=1}^\infty b_n = \infty$ and $\sum\limits_{n=1}^\infty a_n = \infty$;
\item transient if and only if $\sum\limits_{n=1}^\infty b_n < \infty$.
\end{enumerate}
\end{theorem}

\begin{proposition} \label{prop:infinitebar}
In Case (\ref{item:retire}), if the initial state is $i$, then the probability of eventual absorption at zero (that is, of reaching the state $-1$) is
\[
\left(\sum\limits_{n=i}^\infty c_n\right)\bigg/\left(1+\sum\limits_{n=0}^\infty c_n\right),
\]
where $c_n:=\prod\limits_{j=0}^{n}\frac{q_j}{p_j}$. Absorption is transient (that is, occurs with probability $<1$) if $\sum c_n$ converges, is almost sure (that is, occurs with probability $1$) if $\sum c_n$ diverges, and is ergodic (that is, almost sure with finite expected time) if in addition $\sum a_n$ converges.
\end{proposition}

Notice that, by shifting the problem one unit to the right, Proposition~\ref{prop:infinitebar} implies that in Case (\ref{item:forever}), the probability of never reaching $0$ from the initial state $i$ is
\[
\left(\sum\limits_{n=0}^{i-1} b_n\right)\bigg/\left(\sum\limits_{n=0}^\infty b_n\right)
\]
(where $b_0:=1$). The computation of this probability appears at least as far back as \cite[Theorem 2a]{Harris:1952ur}.

We might also be interested in the largest value attained before absorption. As a reminder of the flavour of some of the arguments involved, we will provide the easy proof of the following (which is really the same as \cite[Theorem 2b]{Harris:1952ur}).

\begin{proposition} \label{prop:finitebar}
In Case (\ref{item:retire}), the probability $h_{k,i}$ that, starting at $i\in I = \nn$, the maximum attained value is at most $k\ge1$ is
\[
h_{k,i} =
 \begin{cases}
 0 & \text{ if }\: i\ge k+1\\
 \left(\sum\limits_{n=i}^{k}c_n\right)\bigg/\left(1+\sum\limits_{n=0}^{k}c_n\right) & \text{ if }\: 0\le i\le k,
 \end{cases}
\]
where $c_n:=\prod\limits_{j=0}^{n}\frac{q_j}{p_j}$.
\end{proposition}

\begin{proof}
The event in question is the complement of the event that, starting at $i$, the chain ever hits the set $A=\{k+1,k+2,\dots\}$. Let us write $h'_{k,i}$ for the probability of this latter event (that is, $h'_{k,i}=1-h_{k,i}$). By an application of the Markov property (see \cite[Theorem 1.3.2]{Norris:1998wc}), $(h'_{k,i})_{i\in I}$ is the minimal non-negative solution to
\begin{equation} \label{eqn:hitting}
h'_{k,i} =
 \begin{cases}
 1 & \text{ if }\: i\in A\\
 \sum\limits_{j\in I}\Pi_{ij}h'_{k,j} & \text{ if }\: i\notin A.
 \end{cases}
\end{equation}
Since $-1$ is an absorbing state, we also know that $h'_{k,-1}=0$. So far, we know that $h_{k,-1}=1$ and $h_{k,i}=0$ for $i\ge k+1$. For $0\le i\le k$, we have from (\ref{eqn:hitting}) that
\[
h'_{k,i} = p_ih'_{k,i+1}+q_ih'_{k,i-1}.
\]
Writing $u_{i}=h'_{k,i}-h'_{k,i-1}$ for $0\le i\le k+1$, this gives
\[
u_{i}=\frac{q_{i-1}}{p_{i-1}}u_{i-1}=\frac{q_{i-1}q_{i-2}}{p_{i-1}p_{i-2}}u_{i-2}=\dots=c_{i-1}u_0
\]
(with $c_{-1}:=1$), which implies that
\[
h'_{k,i}=(u_0+\dots+u_i)+h'_{k,-1}=u_0\left(1+\sum\limits_{n=0}^{i-1}c_n\right).
\]
Since $h'_{k,k+1}=1$, it follows that $u_0=\left(1+\sum\limits_{n=0}^{k}c_n\right)^{-1}$. Finally, this gives
\[
h_{k,i}=1-h'_{k,i}=\left(\sum\limits_{n=i}^{k}c_n\right)\bigg/\left(1+\sum\limits_{n=0}^{k}c_n\right)
\]
for $0\le i\le k$.
\end{proof}

\begin{example} \label{ex:drunk}
A case of particular interest is that of constant transition probabilities (that is, for all $i\ge1$, $p_i=p=1-q=1-q_i$). We can interpret this as a one-dimensional drunkard's walk: a drunkard stumbles along a semi-infinite street of constant slope, walking one block to the right with probability $p$ and one block to the left with probability $q$ (with $p<q$, $p=q$ or $p>q$ depending on the sign of the slope). His home is at the origin and the bar is at position $k+1\in\{1,2,\dots\}\cup\{\infty\}$. In Case (\ref{item:forever}), even if he makes it home, he refuses to stay there. By Theorem~\ref{thm:forever}, with probability $1$, these visits home occur:
\begin{enumerate}[(a)]
\item infinitely often with finite mean return time if $p<q$;
\item infinitely often with infinite mean return time if $p=q$;
\item at most finitely often if $p>q$ (in which case, the probability that he starts at $i$ and never reaches home is $1-\left(\frac{q}{p}\right)^i$).
\end{enumerate}
In Case (\ref{item:retire}), if he makes it to his front door (at $0$) he can be persuaded, with probability $q_0>0$ (which for simplicity we assume is equal to $q$), to go inside to bed (at $-1$). On the other hand, if he makes it to the bar then he will never leave. By Proposition~\ref{prop:infinitebar} and Proposition~\ref{prop:finitebar}, the event that, starting at $i$, he eventually goes to bed:
\begin{enumerate}[(a)]
\item (if the bar is at $\infty$) is almost sure if $p\le q$ (with finite expected time if $p<q$ and infinite expected time if $p=q$), and has probability 
\[
\left(\frac{q_0}{p_0}\left(\frac{q}{p}\right)^i\right)\bigg/\left(1-\frac{q}{p}+\frac{q_0}{p_0}\right) = \left(\frac{q}{p}\right)^{i+1}
\]
if $p>q$;
\item (if the bar is at $k+1\in\{i+1,i+2,\dots\}\subseteq\nn$) has probability \[\frac{k+1-i}{k+2}\] if $p=q=q_0$ (and is given by a less appealing expression if $p\ne q$).
\end{enumerate}
\end{example}

\section{UHF algebras} \label{section:uhf}

\subsection{The construction}

Let $m_{-1}=m_0=1$ and let $(m_n)_{n=1}^\infty$ be an enumeration of the primes. Recall that $\Pi=(\Pi_{ij})_{i,j\in\nn\cup\{-1\}}$ is the transition matrix of the Markov chain $(Y_n)_{n\in\nn}$ described in \S~\ref{subsection:mcmodel} (that is, $\Pi_{ij}=\pp(Y_{n+1}=j \mid Y_n=i$)) and $\pi=(\pi_i)_{i\in\nn}$ is its initial distribution (that is, $\pi_i=\pp(Y_0=i)$). In this section, we construct a UHF algebra
\[
M(\Pi,\pi) = M_{m_{Y_0}} \otimes M_{m_{Y_1}} \otimes M_{m_{Y_2}} \otimes \dots.
\]
In Case (\ref{item:forever}) we always obtain an infinite-dimensional $\cs$-algebra, while in Case (\ref{item:retire}) it is possible to end up with a finite matrix algebra.

\subsection{Probabilities}

The following is immediate from Theorem~\ref{thm:forever}, Proposition~\ref{prop:infinitebar}, Proposition~\ref{prop:finitebar} and Example~\ref{ex:drunk}.

\begin{theorem} \label{thm:uhf}~
\begin{enumerate}[(1)]
\item If $q_0=0$, then with probability $1$:
\begin{enumerate}[(a)]
\item $M(\Pi,\pi)$ is isomorphic to the universal UHF algebra $\mathcal{Q}$ if $\sum\limits_{n=1}^\infty\prod\limits_{i=1}^{n}\frac{q_i}{p_i}=\infty$, in particular if $p_i=p\le q=q_i$ for all $i\ge 1$;
\item $M(\Pi,\pi)$ is of `finite type', that is, every prime factor of the supernatural number associated to $M(\Pi,\pi)$ has finite multiplicity, if $\sum\limits_{n=1}^\infty\prod\limits_{i=1}^{n}\frac{q_i}{p_i}<\infty$, in particular if $p_i=p > q=q_i$ for all $i\ge 1$.
\end{enumerate}
\item If $q_0>0$, then the probability that:
\begin{enumerate}[(a)]
\item $M(\Pi,\pi)$ is finite dimensional is $1$ if $\sum\limits_{n=0}^\infty \prod\limits_{j=0}^{n}\frac{q_j}{p_j}=\infty$, in particular if $p_i=p\le q=q_i$ for all $i\ge 1$, and otherwise is
\[
\left(\sum\limits_{i=0}^\infty\pi_i\sum\limits_{n=i}^\infty \prod\limits_{j=0}^{n}\frac{q_j}{p_j}\right)\bigg/\left(1+\sum\limits_{n=0}^\infty \prod\limits_{j=0}^{n}\frac{q_j}{p_j}\right),
\]
which simplifies to
\[
\sum\limits_{i=0}^\infty\pi_i\left(\frac{q}{p}\right)^{i+1}
\]
if $p_i=p> q=q_i$ for all $i\ge 0$;
\item $M(\Pi,\pi)$ is isomorphic to $M_N$, with the highest prime factor of $N$ at most $m_k$, is
\[
\left(\sum\limits_{i=0}^k\pi_i\sum\limits_{n=i}^{k}\prod\limits_{j=0}^{n}\frac{q_j}{p_j}\right)\bigg/\left(1+\sum\limits_{n=0}^{k}\prod\limits_{j=0}^{n}\frac{q_j}{p_j}\right),
\]
which simplifies to
\[
\sum\limits_{i=0}^k\pi_i\frac{k+1-i}{k+2}
\]
if $p_i=p = q=q_i$ for all $i\ge 0$.
\end{enumerate}
\end{enumerate}
\end{theorem}

\subsection{Variations}

The particular random walk we have chosen as our model is of course just one of many possibilities. We might instead for example list the primes as the elements of the grid $\zz^d$ for $d\ge1$. It is a classical theorem of P\'{o}lya \cite{Polya:1921tg} (see also \cite[Chapter \rm{XIV}.7]{Feller:1968wy} or \cite[\S1.6]{Norris:1998wc}) that a simple symmetric walk on this grid (that is, one in which it is only possible to move from a given point to one of the $2d$ neighbouring points, each with equal probability), is recurrent for $d=1$ and $d=2$, but transient for $d\ge3$.

In particular, if our random UHF algebra is constructed according to a drunken walk through an infinite flat city, we still obtain the universal UHF algebra $\mathcal{Q}$. For this reason, one might (and the author personally does) think of $\mathcal{Q}$ as \emph{the drunkard's UHF algebra}. On the other hand, to paraphrase Kakutani \cite[\S3.2]{Durrett:1996tt}, a drunk \emph{bird} (who flies around $\zz^3$) will almost surely produce a UHF algebra of finite type.

\section{Simple limits of point-line algebras} \label{section:ai}

\subsection{Background} \label{subsection:bg}

The Lazar--Lindenstrauss simplex theorem \cite[Theorem 5.2 and its Corollary]{Lazar:1971kx} says that every infinite-dimensional metrisable Choquet simplex $\Delta$ is affinely homeomorphic to a projective limit
\begin{equation} \label{eqn:geom}
\Delta \cong \varprojlim(\Delta_n,\psi_n),
\end{equation}
where for each $n\in\nn$, $\Delta_n$ is an $n$-dimensional simplex and $\psi_n\colon\Delta_{n}\to\Delta_{n-1}$ is affine and surjective. Equivalently, the space $\aff(\Delta)$ of continuous affine maps $\Delta\to\rr$ is isomorphic  to the limit
\begin{equation} \label{eqn:alg}
\aff(\Delta) \cong \varinjlim(\aff(\Delta_{n-1}),\psi_n^*) \cong \varinjlim(\rr^{n},\psi_n^*)
\end{equation}
in the category of complete order unit spaces (in which the morphisms are positive unital linear maps). Here, we identify $\aff(\Delta_{n-1})$ with $\rr^n$ via the basis $(f_{j,n})_{0\le j\le n-1}$ that is defined on the vertices $e_{0,n-1},e_{1,n-1},\dots,e_{n-1,n-1}$ of $\Delta_{n-1}$ by
\begin{equation} \label{eqn:dual}
f_{j,n}(e_{i,n})=\delta_{ij}.
\end{equation}
We will refer to these equivalent statements of the theorem as, respectively, the geometric version and the algebraic one.

Both versions have proved useful in the construction of simple inductive limit $\cs$-algebras $A=\varinjlim A_n$ with prescribed tracial state space $T(A)$. We would in particular like to highlight the approximately homogeneous (AH) cases $A_n=p_nC(X_n,M_{m_n})p_n$, where:
\begin{itemize}
\item[\cite{Goodearl:1977dq,Blackadar:1980zr}] \label{item:af} $X_n=\{1,\dots,l_n\}$ (so that $A$ is an approximately finite-dimensional (AF) algebra);
\item[\cite{Thomsen:1994qy}] \label{item:ai} $X_n=[0,1]$ (so that $A$ is an approximately interval (AI) algebra).
\end{itemize}
The AF construction uses the geometric version (\ref{eqn:geom})of Lazar--Lindenstrauss. The AI construction uses the algebraic one (\ref{eqn:alg}), together with an intertwining \cite[Lemma 3.8]{Thomsen:1994qy}
\[
\begin{tikzcd}
C_\mathbb{R}([0,1]) \arrow[r,"\varphi_1"] \arrow[d] & C_\mathbb{R}([0,1]) \arrow[r,"\varphi_2"] \arrow[d] & \dots \arrow[r] & \aff(T(A)) \arrow[d,dashed]\\
\rr^1 \arrow[r,"\psi_1^*"] \arrow[ur] & \rr^2 \arrow[r,"\psi_2^*"] \arrow[ur] & \dots \arrow[r] & \aff(\Delta)
\end{tikzcd}
\]
(where $C_\mathbb{R}([0,1])$ is identified with $\aff(T(C([0,1],M_{m_n})))$ via the embedding of $C([0,1])$ into the centre of $C([0,1],M_{m_n})$) and a suitable Krein--Milman theorem \cite[Theorem 2.1]{Thomsen:1994qy} to approximate each positive unital linear map $\varphi_n$ by an average of maps induced by $^*$-homomorphisms.

Simple AI algebras are particularly noteworthy for demonstrating the necessity for a classifying invariant for simple amenable $\cs$-algebras to include not just traces and ordered $K$-theory but also the pairing between the two (see \cite[p.~29]{Rordam:2002yu}). For us though, the key point is that the construction of these AH algebras is via an \emph{algorithm} whose
\begin{itemize}
\item[\emph{input}] \label{item:seq} is a sequence of affine surjections $\psi_n\colon\Delta_n\to\Delta_{n-1}$ (or positive unital linear maps $\psi_n^*\colon\rr^n\to\rr^{n+1}$), and whose
\item[\emph{output}] is a simple $\cs$-algebra whose tracial simplex is specified by (\ref{eqn:geom}) or (\ref{eqn:alg}).
\end{itemize}
In \S~\ref{subsection:rm}, we will use `representing matrices' to turn the space of sequences (\ref{item:seq}) into a measure space in which we can compute probabilities. First though, for the sake of noise reduction we switch our attention from homogeneous to subhomogeneous building blocks.

The introduction of boundary conditions at the endpoints of the interval (obtaining what are variously referred to as \emph{point-line algebras}, \emph{Elliott-Thomsen building blocks} or \emph{one-dimensional noncommutative CW complexes}) provides access to a wider range for the $K$-theory of the limit algebra. (To exhaust the full range, in particular to account for torsion in $K_0$, one must allow for slightly higher dimensional base spaces; see \cite{Elliott:1996qe}.) That said, our interest in this section is altogether to rid ourselves of this turbulent $K$-theory, and focus on a class for which the only thing that matters is traces, namely, $\cs$-algebras built from prime dimension drop algebras
\[
Z_{x,y} = \{f\in C([0,1],M_x\otimes M_y) \mid f(0)\in M_x\otimes 1_y,\: f(1)\in 1_x\otimes M_y\}, \quad (x,y)=1.
\]
By \cite[Theorems 4.5 and 6.2]{Jiang:1999hb}, simple inductive limits of these building blocks are completely classified by the tracial simplex.  The `existence' part of the classification is still based on Thomsen's algorithm, with a suitably refined Krein--Milman theorem \cite[Theorem 2.1]{Li:1999aa}.

\subsection{Representing matrices} \label{subsection:rm}

By reordering the set $(e_{i,n})_{0\le i\le n}$ of vertices of $\Delta_n$ if necessary, we can visualise (\ref{eqn:geom}) as follows: $\Delta_{n}$ sits atop its base $\Delta_{n-1}$, and $\psi_n$ is the collapsing map that fixes the base and sends $e_{n,n}$ to $\sum\limits_{i=1}^{n}a_{i,n}e_{i-1,n-1}\in\Delta_{n-1}$. Dually, $\psi_n^*\colon\rr^{n}\cong\aff(\Delta_{n-1})\to\aff(\Delta_n)\cong\rr^{n+1}$ has matrix
\begin{equation} \label{eqn:algrep}
\begin{pmatrix}
1 & ~ & ~ & ~\\
~ & 1 & ~ & ~\\
~ & ~ & \ddots & ~\\
~ & ~ & ~ & 1\\
a_{1,n} & a_{2,n} & \dots & a_{n,n}
\end{pmatrix}.
\end{equation}
The triangular matrix $(a_{i,n})_{1\le i\le n,\,n\ge1}$ is called a \emph{representing matrix} for the simplex $\Delta$. Assigning to a representing matrix $A$ the simplex $\Delta_A$ defined by (\ref{eqn:algrep}) is a well-defined function, but is $\infty$-to-$1$: for a start, \cite[Theorem 4.7]{Lindenstrauss:1978vi} says that a sufficient condition for matrices $A$ and $B$ to represent the same simplex is that they have the same asymptotic behaviour, that is, satisfy $\sum\limits_{n=1}^\infty\sum\limits_{i=1}^n |a_{i,n}-b_{i,n}|<\infty$. The full story is more complicated, because this condition is certainly not necessary. For example, it can be shown (see \cite[p.~317 and Theorem 4.3]{Sternfeld:1980vo}) that the representing matrices for which:
\begin{equation} \label{eqn:k0}
(a)\quad a_{n,n}=1 \quad (b)\quad a_{1,n}=1 \quad (c)\quad a_{i,n}=\frac{1}{n} \quad\text{ for } n\ge1 \text{ and } 1\le i\le n,
\end{equation}
all give rise to the Bauer simplex $K$ whose extreme boundary is homeomorphic to $\{\frac{1}{n}\}_{n=1}^\infty\cup\{0\}$.

In the random constructions to follow, we use
\begin{equation}
(a_{i,j})_{i,j\ge1} \mapsto \left(\sum\limits_{i=1}^na_{in}e_{i-1,n-1}\right)_{n\ge1}
\end{equation}
to identify the set $R$ of representing matrices with $\prod\limits_{n\in\nn}\Delta_n$. Probabilities in the set $\Delta(R)$ of all metrisable Choquet simplexes (which, as in \cite{Edwards:2011vj}, can be viewed as a subset of the unit sphere of $\ell_1$ with its $w^*$-topology) are computed by pushing forward a choice of product measure $\gamma=\otimes_{n\in\nn}\gamma_n$ on $R$. We will consider three possibilities for $\gamma$:
\begin{itemize}
\item[$(K)$]\label{item:con} $\gamma_n$ the point mass at $\sum_{i=0}^{n}\frac{1}{n}e_{i,n}$;
\item[$(C)$]\label{item:can} $\gamma_n$ the uniform measure on the $n+1$ vertices of $\Delta_n$;
\item[$(P)$]\label{item:pou} $\gamma_{s_n+i}$ normalised Lebesgue measure on $\conv\{e_{0,n},\dots,e_{n-i+1,n}\}\subseteq\Delta_n\subseteq\Delta_{s_{n}+i}$ for $1\le i\le n+1$, where $s_n=0+1+\dots+n=\frac{n(n+1)}{2}$.
\end{itemize}

\subsection{The constructions}

We use the Markov chain $(Y_n)_{n\in\nn}$ of \S~\ref{subsection:mcmodel}, a choice of probability measure $\gamma$ of the form $(K)$, $(C)$ or $(P)$ as described in \S~\ref{subsection:rm}, and Thomsen's algorithm \S\ref{subsection:bg}, to build a simple inductive limit
\[
Z(\Pi,\pi,\gamma) =  \varinjlim(Z_{x_n,y_n},\varphi_n)
\]
whose tracial state space is affinely homeomorphic to $\varprojlim(\Delta_{Y_n},\psi_{n})$. Here, $\Delta_{-1}:=\Delta_0$ and the connecting map $\psi_{n}\colon\Delta_{Y_n}\to\Delta_{Y_{n-1}}$ is either:
\begin{enumerate}[(i)]
\item the standard inclusion (that is, assigns the vertices of $\Delta_{Y_n}$ to the first $Y_n+1$ vertices of $\Delta_{Y_{n-1}}$) if $Y_{n}=Y_{n-1}-1$; or
\item chosen randomly according to the probability measure $\gamma$ if $Y_{n}=Y_{n-1}+1$.
\end{enumerate}

\subsection{Probabilities}

\begin{theorem} \label{thm:js}
\begin{enumerate}[(1)]
\item \label{item:jsa} If $q_0=0$, then with probability $1$, $Z(\Pi,\pi,\gamma)$: 
\begin{enumerate}[(a)]
\item \label{item:jsa1} is monotracial (hence, is isomorphic to the Jiang--Su algebra $\js$) if $\sum\limits_{n=1}^\infty\prod\limits_{i=1}^{n}\frac{q_i}{p_i}=\infty$, in particular if $p_i=p\le q=q_i$ for all $i\ge 1$;
\item \label{item:jsa2} has infinite-dimensional tracial state space if $\sum\limits_{n=1}^\infty\prod\limits_{i=1}^{n}\frac{q_i}{p_i}<\infty$, in particular if $p_i=p > q=q_i$ for all $i\ge 1$; moreover, in this case $T(Z(\Pi,\pi,\gamma))$ is with probability $1$ affinely homeomorphic to either
\begin{itemize}
\item[$K$] the Bauer simplex whose extremal boundary is $\{\frac{1}{n}\}_{n=1}^\infty\cup\{0\}$ or
\item[$C$] the Bauer simplex whose extremal boundary is the Cantor set or
\item[$P$] the Poulsen simplex.
\end{itemize}
\end{enumerate}
\item \label{item:jsb} If $q_0>0$, then the probability that $Z(\Pi,\pi,\gamma)$ is isomorphic to $\js$ is $1$ if $\sum\limits_{n=0}^\infty \prod\limits_{j=0}^{n}\frac{q_j}{p_j}=\infty$, in particular if $p_i=p\le q=q_i$ for all $i\ge 1$, and otherwise is
\[
\left(\sum\limits_{i=0}^\infty\pi_i\sum\limits_{n=i}^\infty \prod\limits_{j=0}^{n}\frac{q_j}{p_j}\right)\bigg/\left(1+\sum\limits_{n=0}^\infty \prod\limits_{j=0}^{n}\frac{q_j}{p_j}\right),
\]
which simplifies to
\[
\sum\limits_{i=0}^\infty\pi_i\left(\frac{q}{p}\right)^{i+1}
\]
if $p_i=p> q=q_i$ for all $i\ge 0$.
\end{enumerate}
\end{theorem}

\begin{proof}
(\ref{item:jsa}) First, since an inverse limit is isomorphic to the limit over any cofinal subset of the index set, $Z(\Pi,\pi,\gamma)$ is monotracial precisely when the walk $(Y_n)$ visits the states $\{0,-1\}$ infinitely often. This demonstrates (\ref{item:jsa1}) and the first assertion of (\ref{item:jsa2}).

Case $K$ is not just almost sure, but is the uniquely determined possibility (see (\ref{eqn:k0})).

Case $C$ follows from Brouwer's Theorem \cite[Theorem 3]{Brouwer:1910wa}, which characterises the Cantor set up to homeomorphism as the unique nonempty totally disconnected compact metrisable space without isolated points. By \cite[Theorem 5.1]{Lazar:1971kx}, $T(Z(\Pi,\pi,\gamma))$ is a Bauer simplex with totally disconnected extremal boundary. Essentially, the dual bases $(f_{j,n})_{0\le j\le n-1}$ of (\ref{eqn:dual}) correspond to partitions of unity over a decreasing sequence of clopen covers of the boundary, with the representing matrix describing the way that successive partitions sit inside each other (at each stage, the subsequent partition is obtained from the current one by splitting one element into two pieces). We might picture the random process as a population tree, where at each level, one member of the population is chosen at random to reproduce. In this language, the probability that $\partial_eT(Z(\Pi,\pi,\gamma))$ has an isolated point (or equivalently, is not homeomorphic to the Cantor set) is the probability that some branch of the tree eventually becomes an evolutionary dead end, which is at most
\[
\sum_{n=2}^\infty\left( n\cdot\prod_{i=n}^{\infty}\frac{i-1}{i} \right) = \sum_{n=2}^\infty n\cdot\lim_{m\to\infty} \frac{n-1}{m} = 0. 
\]

Case $P$ similarly follows from (the remark after) \cite[Theorem 5.6]{Lazar:1971kx}. A representing matrix yields the Poulsen simplex if the vectors $(a_{1n},\dots,a_{nn},0,\dots)$ are dense in the positive face of the unit sphere of $\ell_1$. This fails to happen only if there exist $\varepsilon>0$, $n\ge 1$ and some barycentric coordinates $(x_0,\dots,x_n)$ such that each choice of random point in $\conv\{e_{0,m},\dots,e_{n,m}\}$, $m\ge n$, misses the $\ell_1$-ball $B$ centred at $\sum\limits_{i=0}^nx_ie_{i,m}$ with $\gamma(B)=\varepsilon$. The probability of this is at most $\prod_{m\ge n}(1-\varepsilon)$, which is $0$.

Finally, (\ref{item:jsb}) follows from Example~\ref{ex:drunk}.
\end{proof}

\begin{remark} \label{remark:simplex}~
\begin{enumerate}
\item To allow for the possibility of a nontrivial finite-dimensional trace space, we might in the case $q_0>0$ take the dimension of the randomly constructed simplex to be $\sup\{Y(n) \mid n\in\nn\}$. As in Example~\ref{ex:drunk}, supposing that $p_i=p=q_i$ for all $i$, the probability that $Z(\Pi,\pi,\gamma)$ has at most $k\ge1$ extremal traces is then
\[
\sum\limits_{i=0}^k\pi_i\frac{k-i}{k+1}.
\]
\item We could equally well have opted for the stably projectionless versions of dimension drop algebras, so-called `Razak blocks', simple inductive limits of which (assuming `continuous scale') are completely classified by the tracial state space \cite{Razak:2002kq,Tsang:2005fj}. In this case, the almost-sure monotracial limit obtained in Theorem~\ref{thm:js} would be the $\cs$-algebra $\mathcal{W}$ instead of $\js$.
\end{enumerate}
\end{remark}

\section{Villadsen algebras of the first type} \label{section:ah}

By once again making a random choice at each step of an inductive limit, we this time build an approximately homogeneous (AH) algebra, the radius of comparison of which is a random variable whose distribution is the subject of this section.

\subsection{Background}

The radius of comparison is a numerical invariant of the Cuntz semigroup that was used in \cite{Toms:2008hl} to distinguish $\cs$-algebras with the same Elliott invariant. In \cite[Theorem 5.1]{Toms:2006up}, it is shown how to construct simple, unital AH algebras with arbitrary `dimension-rank ratio' $c$. Moreover, as pointed out in \cite[Corollary 5.2]{Asadi:2021th}, an algebra constructed in this way has radius of comparison $r=\frac{c}{2}$.

These algebras of Toms are, in the language of \cite{Toms:2009fk}, \emph{Villadsen algebras of the first type}. That is, they are inductive limits $A_X=\varinjlim(M_{m_i}\otimes C(X^{\times n_i}), \varphi_i)$, with a fixed compact Hausdorff $X$ (the `seed space') and each connecting map $\varphi_i$ of the form $\varphi_i(f)=\diag(f\circ\xi_{i,1},\dots,f\circ\xi_{i,m_{i+1}/m_i})$, where each $\xi_{i,j}\colon X^{\times n_{i+1}} \to X^{\times n_i}$ is either constant or a coordinate projection. Following  \cite{Toms:2006up}, we will take $X$ to be the sphere $S^2$, but as in \cite[\S8]{Toms:2009fk}, any finite (but nonzero) dimensional CW complex would do.

\subsection{The construction} \label{subsection:drunk}

This time, we take a random walk not on the natural numbers but on an infinite complete binary tree (or rather, on the countable disjoint union of copies of this tree). The initial position $W_0$ of the walk is specified by a probability distribution $\pi$ on $I=\{0\}\cup\{2^k\mid k\in\zz\}$. Once the starting point is decided, the walk constructs a random binary number, that is, determines the value of a random variable
\begin{equation} \label{eqn:binary}
W = 0.W_1W_2 \dots = \sum_{i=1}^\infty \frac{W_i}{2^i},
\end{equation}
where $\{W_i\}_{i\in\nn}$ are independent, as follows. After $i$ steps, our position is on the $i$th level of the tree. We descend to the level below by taking either the left branch (with probability $p_{i+1}$), in which case $W_{i+1}=0$, or the right branch (with probability $q_{i+1}=1-p_{i+1}$), in which case $W_{i+1}=1$. Just as in (\ref{eqn:tm}), the transition probabilities can be encoded by a matrix $\Pi$.

By following the procedure outlined in the proof of \cite[Theorem 5.1]{Toms:2006up}, we use the random variables $\{W_i\}_{i\in\nn}$ to construct a Villadsen algebra of the first type
\[
B(\Pi,\pi) = \varinjlim(B_i,\varphi_i).
\]
Here, $B_i=M_{n_i}(C(T_i))$, where $T_i=(S^2)^{m_0m_1\dots m_i}$, and the connecting map $\varphi_i\colon B_i\to B_{i+1}$ is of the form
\[
\varphi_i(f)(t) = \diag\left(f\circ\pi_i^1(t),\dots,f\circ\pi_i^{m_{i+1}}(t),f\left(t_i^1\right),\dots,f\left(t_i^{s_{i+1}}\right)\right),
\]
where $\pi_i^j\colon T_{i+1}=T_i^{m_{i+1}}\to T_i$, $1\le j\le m_{i+1}$, are the coordinate projections, and the points $t_i^j$ are chosen to make the limit simple (as in \cite{Villadsen:1998ys}). There is some flexibility in the choice of the natural numbers $m_i$ and $s_i$ (which determine $n_i=(m_i+s_i)n_{i-1}$). The ratios $\frac{m_i}{m_i+s_i}$ are what govern the radius of comparison $R$ of $B(\Pi,\pi)$, and these ratios are dictated by the random walk. We set $\frac{m_0}{n_0}=W_0$, and $\frac{m_i}{m_i+s_i}=2^{-\frac{W_i}{2^i}}$ for $i\ge1$. In other words, we relate
\[
R = \lim_{i\to\infty}\frac{\dim T_{i}}{2n_{i}} = \lim_{i\to\infty}\frac{m_0m_1\dots m_{i}}{n_{i}} = \lim_{i\to\infty}\frac{m_0m_1\dots m_{i-1}}{n_{i-1}} \frac{m_{i}}{m_{i}+s_{i}} = \prod_{i\in\nn}\frac{m_i}{m_i+s_i}
\]
to $\{W_i\}_{i\in\nn}$ by
\begin{equation} \label{eqn:roc}
R = W_0\cdot2^{-W}.
\end{equation}

By the Jessen--Wintner law of pure types (see \cite[Theorem 35]{Jessen:1935ux}, or \cite[Theorem 3.26]{Breiman:1992vp} for an elementary treatment), the distribution of $W$ is either:
\begin{enumerate}[(a)]
\item \label{item:d} discrete (that is, there is a countable set $J$ such that $\pp(W\in J) =1$); or
\item \label{item:s} singular (that is, $\pp(W=w)=0$ for every $w\in\rr$, but there is a Borel set $B$ of Lebesgue measure zero such that $\pp(W\in B) =1$); or
\item \label{item:a} absolutely continuous with respect to Lebesgue (so has a density $f$).
\end{enumerate}
(Distributions can in general be of mixed type, but this is not the case for a convergent series of discrete random variables.)

For the sake of computing the expected value of $R$, we will assume that we are in case (\ref{item:a}). By the results of \cite{Marsaglia:1971uq} (see also \cite{Chatterji:1964vn}), this occurs precisely when there is $\beta\in\rr$ such that, for every $n\ge1$,
\[
p_n = \frac{1}{1+e^{\beta/2^n}} \quad\text{and}\quad q_n = \frac{e^{\beta/2^n}}{1+e^{\beta/2^n}}.
\]
Then, the probability density function of $W$ is
\begin{equation} \label{eqn:pdf}
f_\beta(x) = \frac{\beta e^{\beta x}}{e^\beta-1} \quad,\quad 0\le x\le 1
\end{equation}
if $\beta\ne0$, or $f_0=1$, that is, representing the uniform distribution, if $\beta=0$.  

\subsection{Probabilities}

Within the framework of \S~\ref{subsection:drunk}, we can compute the expected value of the radius of comparison of the random $\cs$-algebra $B(\Pi,\pi)$.

\begin{theorem} \label{thm:roc}
\begin{enumerate}
\item \label{item:roca} With probability $\pi_0$, $B(\Pi,\pi)$ is $\js$-stable (in fact, a UHF algebra).
\item \label{item:rocb} For every $r>0$,
\[
\pp(R\ge r) = \sum_{k\in \zz}\pi_{2^k}G_\beta\left(\frac{r}{2^k}\right),
\]
where
\[
G_\beta(x) = \begin{cases}
 1 & \:\text{ if }\: x\le\frac{1}{2}\\
 \frac{x^{-\beta/\ln2}-1}{e^\beta-1} & \:\text{ if }\: \frac{1}{2}<x<1\\
 0 & \:\text{ if }\: x\ge1
 \end{cases}
\]
if $\beta\ne0$, and
\[
G_0(x) = \begin{cases}
 1 & \:\text{ if }\: x\le\frac{1}{2}\\
 -\frac{\ln x}{\ln2} & \:\text{ if }\: \frac{1}{2}<x<1\\
 0 & \:\text{ if }\: x\ge1.
 \end{cases}
\]
\item \label{item:rocc} The expected value of $R$ satisfies
\[
\frac{\ee(R)}{\sum\limits_{k\in \zz}2^k\pi_{2^k}} = \begin{cases}
 \frac{1}{\ln4} & \:\text{ if }\: \beta=0\\
 \ln2 & \:\text{ if }\: \beta=\ln2\\
 \frac{\beta(e^\beta-2)}{2(e^\beta-1)(\beta-\ln2)} & \:\text{ if }\: \beta\notin\{0,\ln2\}
 \end{cases}
\]
if $\sum\limits_{k\in \nn}2^k\pi_{2^k}$ is finite. Otherwise, $\ee(R)=\infty$.
\end{enumerate}
\end{theorem}

\begin{proof}
(\ref{item:roca}) follows from the fact that $\js$-stability implies strict comparison of positive elements \cite[Corollary 4.6]{Rordam:2004kq}, or in other words radius of comparison $R=0$. With the present construction, this is only possible with initial value $W_0=0$, in which case each $X_i$ is a point, so $B(\Pi,\pi)$ is a UHF algebra (which is $\js$-stable).

(\ref{item:rocb}) can be deduced using the probability density function $f_\beta$ of (\ref{eqn:pdf}) and independence of $W_0$ and $W$. From (\ref{eqn:pdf}), we compute the (complementary) cumulative distribution function of $2^{-W}$: for $\beta\ne0$ and $\frac{1}{2}\le x\le 1$ we have  
\[
\pp(2^{-W} \ge x) = \pp\left(W\le\log_2\frac{1}{x}\right)
= \int_0^{\log_2\frac{1}{x}} \frac{\beta e^{\beta t}}{e^\beta-1}dt
= \frac{x^{-\beta/\ln2}-1}{e^\beta-1}
= G_\beta(x).
\]
If $\beta=0$, this probability is
\[
\pp(2^{-W} \ge x) = \int_0^{\log_2\frac{1}{x}} 1\:dt
= \log_2\frac{1}{x}
= -\frac{\ln x}{\ln2}
= G_0(x).
\]
Then, for $r>0$ we have
\begin{align*}
\pp(R\ge r) &= \pp(W_0\cdot 2^{-W} \ge r)\\
&= \sum_{k\in\zz}\pp(W_0=2^k, W_0\cdot 2^{-W} \ge r)\\
&= \sum_{k\in\zz}\pi_{2^k} \pp\left(2^{-W} \ge \frac{r}{2^k}\right)\\
&= \sum_{k\in \zz}\pi_{2^k}G_\beta\left(\frac{r}{2^k}\right).
\end{align*}
(\ref{item:rocc}) now follows from (\ref{item:rocb}), thanks to the well-known formula
\[
\ee(R) = \int_0^\infty \pp(R\ge r) dr
\]
(see for example \cite[V.6, Lemma 1]{Feller:1971tc}). Using (\ref{item:rocb}) and the monotone convergence theorem, we can compute this as
\[
\ee(R) = \sum_{k\in \zz}\pi_{2^k}\int_0^\infty G_\beta\left(\frac{r}{2^k}\right) dr.
\]
If $\beta=0$, this is
\begin{align*}
\ee(R) &= \sum_{k\in \zz}\pi_{2^k}\left(2^{k-1}+\int_{2^{k-1}}^{2^k} -\frac{\ln r/2^k}{\ln2} dr\right)\\
&= \sum_{k\in \zz}2^k\pi_{2^k}\left(\frac{1}{2}+\int_{\frac{1}{2}}^{1} -\frac{\ln t}{\ln2} dt\right)\\
&= \sum_{k\in \zz}2^k\pi_{2^k}\frac{1}{\ln4}.
\end{align*}
If $\beta=\ln2$, we have
\begin{align*}
\ee(R) &= \sum_{k\in \zz}\pi_{2^k}\left(2^{k-1}+\int_{2^{k-1}}^{2^k}\left(\frac{2^k}{r}-1\right)dr\right)\\
&= \sum_{k\in \zz}2^k\pi_{2^k}\left(\frac{1}{2}+\int_{\frac{1}{2}}^{1}\left(\frac{1}{t}-1\right)dt\right)\\
&= \sum_{k\in \zz}2^k\pi_{2^k}\ln2.
\end{align*}
Finally, if $\beta\notin\{0,\ln2\}$, we have
\begin{align*}
\ee(R) &= \sum_{k\in \zz}\pi_{2^k}\left(2^{k-1}+\int_{2^{k-1}}^{2^k}\left(\frac{(r/2^k)^{-\beta/\ln2}-1}{e^\beta-1}\right)dr\right)\\
&= \sum_{k\in \zz}2^k\pi_{2^k}\left(\frac{1}{2}+\int_{\frac{1}{2}}^{1}\left(\frac{t^{-\beta/\ln2}-1}{e^\beta-1}\right)dt\right)\\
&= \sum_{k\in \zz}2^k\pi_{2^k}\frac{\beta(e^\beta-2)}{2(e^\beta-1)(\beta-\ln2)}. \qedhere
\end{align*} 
\end{proof}

\subsection{Variations}

In an alternative scenario, we proceed as follows. The initial value $\frac{m_0}{n_0}$ is once again specified by a probability distribution $\pi$, let us say on the nonzero natural numbers (although this is immaterial for the ensuing discussion). The construction is essentially the same as in \S~\ref{subsection:drunk}. The connecting maps $\varphi_i\colon B_i\to B_{i+1}$ are of the form
\[
\varphi_i(f)(t) = \diag\left(f\circ\pi_i^1(t),\dots,f\circ\pi_i^{m_{i+1}}(t),f\left(t_i^1\right),\dots,f\left(t_i^{i+1}\right)\right),
\]
the point evaluations chosen to ensure simplicity of the limit. But this time, there are two possibilities:
\begin{enumerate}[(a)]
\item with probability $p_i$, we make a tame choice, that is, we set $m_{i}=1$, so that $\frac{m_{i}}{m_{i}+i} = \frac{1}{i+1}$; or
\item with probability $q_i=1-p_i$, we make a more exotic choice, that is, we set $m_{i}=i(2^i-1)$, so that $\frac{m_{i}}{m_{i}+i} = 1-\frac{1}{2^i}$.
\end{enumerate}
The random inductive limit constructed this way is $\js$-stable if and only if we make the tame choice infinitely often, which occurs with probability
\[
1-\sum_{j=0}^\infty p_{j}\prod_{i\ge j+1} q_i
\]
(where $p_0=1$). This means that, if for example $q_i=\frac{1}{2}$ for every $i$, then we almost surely construct a $\js$-stable $\cs$-algebra.

Suppose on the other hand that, while $q_1=\frac{1}{2}$, subsequently $q_i=1-\frac{1}{i^2}$ for $i\ge2$ (say). Then, using the fact that
\[
\prod_{i\ge m} \left(1-\frac{1}{i^2}\right) = \lim_{k\to\infty}\prod_{i=m}^k\frac{i-1}{i}\cdot\frac{i+1}{i} = \lim_{k\to\infty}\frac{m-1}{m}\cdot\frac{k+1}{k} = \frac{m-1}{m},
\]
the probability of a $\js$-stable limit is
\[
1-\left(\frac{2-1}{2} + \sum_{j=2}^\infty\frac{1}{j^2}\cdot\frac{j}{j+1}\right) = 1-\left(\frac{1}{2}+\frac{1}{2}\right) = 0,
\]
so our random $\cs$-algebra is almost surely \emph{not} $\js$-stable.

\section{Graph $\cs$-algebras} \label{section:pi}

By repeatedly creating random pairings, we construct a random $r$-regular multigraph $\hat\g_{n,r}$ on $n$ labelled vertices, then generate the associated graph $\cs$-algebra $\cs(\hat\g_{n,r})$. In this section, we investigate the asymptotic distribution of the $K_0$-group of this algebra.

\subsection{Background}

A (countable) directed graph $E$ consists of a vertex set $E^0$, an edge set $E^1$, and range and source maps $r,s\colon E^1\to E^0$. A \emph{path} in $E$ is a (possibly finite) sequence of edges $(\alpha_i)_{i\ge1}$ such that $r(\alpha_i)=s(\alpha_{i+1})$ for every $i$, and a \emph{loop} is a finite path whose initial and final vertices coincide. A loop  $\alpha=(\alpha_1,\dots,\alpha_n)$ \emph{has an exit} if there exist $e\in E^1$ and $i\in\{1,\dots,n\}$ such that $s(e)=s(\alpha_i)$ but $e\ne \alpha_i$. A vertex $v\in E^0$ is called:
\begin{enumerate}[(a)]
\item a \emph{source} if $r^{-1}(v)=\emptyset$;
\item a \emph{sink} if $s^{-1}(v)=\emptyset$;
\item an \emph{infinite emitter} if $|s^{-1}(v)|=\infty$;
\item \emph{cofinal} if it can be connected to some point in any given infinite path in $E$.
\end{enumerate}

A \emph{Cuntz--Krieger $E$-family} associated to a finite directed graph $E=(E^0,E^1,s,r)$ is a set
\begin{equation} \label{eqn:gen}
\{p_v \mid v\in E^0\} \cup \{s_e \mid e\in E^1\},
\end{equation}
where the $p_v$ are mutually orthogonal projections and the $s_e$ are partial isometries satisfying:
\begin{align} \label{eqn:rel}
s_e^*s_f &= 0 &&\forall\: e,f\in E^1 \text{ with } e\ne f \nonumber\\
s_e^*s_e &= p_{r(e)} &&\forall\: e\in E^1 \nonumber\\
s_es_e^* &\le p_{s(e)} &&\forall\: e\in E^1 \nonumber\\
p_v &= \sum\limits_{e\in s^{-1}(v)} s_es_e^* &&\forall\: v\in E^0 \:\text{that is not a sink}.
\end{align}
(For countably infinite graphs, which we do not consider here, the last equation should hold for vertices $v$ that are neither sinks nor infinite emitters.)

The \emph{graph algebra} $\cs(E)$ is the universal $\cs$-algebra with generators (\ref{eqn:gen}) satisfying the relations (\ref{eqn:rel}). A \emph{Cuntz--Krieger algebra} may be defined as the graph $\cs$-algebra of a finite graph without sinks or sources. By \cite[\S4]{Kumjian:1997uq}, this is equivalent to the original definition \cite{Cuntz:1980hl}. In fact (see \cite[Theorem 3.12]{Arklint:2015tu}), a graph algebra $\cs(E)$ is isomorphic to a Cuntz--Krieger algebra if and only if $E$ is a finite graph with no sinks, or equivalently if $\cs(E)$ is unital and 
\[
\rank(K_0(\cs(E))) = \rank(K_1(\cs(E))).
\]
In general, if $E$ is a finite graph with no sinks, and $A_E$ is its adjacency matrix
\[
A_E(i,j) = \left|\left\{e\in E^1 \mid s(e)=i, r(e)=j\right\}\right|,
\]
then $K_*(\cs(E))$ is given by
\begin{equation} \label{eqn:kt}
K_0(\cs(E)) \cong \coker(A_E^t-I) \quad\text{and}\quad K_1(\cs(E)) \cong \ker(A_E^t-I).
\end{equation}
(See \cite[Theorem 3.2]{Raeburn:2004tg}, where this is proved for row-finite graphs possibly with sinks, and also \cite[Theorem 3.1]{Drinen:2002kx}, which considers arbitrary graphs.)

Finally, we recall conditions on $E$ that correspond to $\cs(E)$ being a Kirchberg algebra (see \cite[Propositions 5.1 and 5.3]{Bates:2000fk} and \cite[Corollary 3.10]{Kumjian:1998nr}).

\begin{proposition} \label{prop:pisun}
Let $E$ be a finite directed graph.
\begin{enumerate}[(1)]
\item \label{item:pi1} If every vertex connects to a loop and every loop has an exit, then $\cs(E)$ is purely infinite.
\item \label{item:pi2} If $E$ has no sinks, then $\cs(E)$ is simple if and only if every loop in $E$ has an exit and every vertex in $E$ is cofinal. In this case, $\cs(E)$ is an AF algebra if $E$ has no loops, and otherwise is purely infinite.
\end{enumerate}
\end{proposition}

\subsection{The construction}

For a given natural number $r$, a multigraph $E$ is said to be \emph{$r$-regular} if every vertex has degree $r$. There are several probability distributions that can be used to construct a random $r$-regular (undirected, multi-) graph on $n\in2\nn$ labelled vertices. We will use the \emph{perfect matchings model}
\[
\hat\g_{n,r}=\underbrace{\g_{n,1}+\dots+\g_{n,1}}_r,
\]
that is, the union of $r$ independent, uniformly random perfect matchings. One might also consider the \emph{uniform model} $\g'_{n,r}$, that is, a random element of the set of $r$-regular multigraphs with the uniform distribution, conditioned on there being no loops. By \cite[Theorem 11]{Janson:1995wq}, these two models are \emph{contiguous}, which means that a sequence of events that occurs \emph{asymptotically almost surely} (that is, with probability approaching $1$ as $n\to\infty$) according to one distribution, occurs asymptotically almost surely according to the other.

We convert the undirected graph $\hat\g_{n,r}$ into a directed one by replacing each edge (between given vertices $i,j$) by two (one from $i$ to $j$ and one from $j$ to $i$).

\begin{remark}
One can expect strong connectivity in random regular graphs. Conditioning $\g'_{n,r}$ on there being no multiple edges, one obtains (up to contiguity) the random $r$-regular graph $\g_{n,r}$. By work of Bollob\'as \cite{Bollobas:1980wo} (see \cite[Theorem 11.9]{Frieze:2016tw}), $\g_{n,r}$ is asymptotically almost surely (a.a.s.) $r$-connected. For our purposes, it will be enough for us to know that $\hat\g_{n,r}$ is a.a.s.\ ($1$-) connected, which is not difficult to show (see \cite{Lavrov:2022se}).
\end{remark}

\subsection{Probabilities}

\begin{theorem} \label{thm:ck}
For every even integer $n$ and every integer $r\ge3$, $\cs(\hat\g_{n,r})$ is isomorphic to a Cuntz-Krieger algebra that is purely infinite and is asymptotically almost surely simple. Moreover, for any finite set $P$ of odd primes not dividing $r-1$, and any finite abelian group $V$ such that $|V|$ is in the multiplicative semigroup generated by $P\subseteq\nn$,
\begin{equation} \label{eqn:kd}
\lim_{n\in2\nn}\pp\left(K_0\left(\cs(\hat\g_{n,r})\right)\otimes\prod_{p\in P}\zz_p\cong V\right) = N(V)\prod_{p\in P}\prod_{k\ge0}\left(1-p^{-2k-1}\right),
\end{equation}
where $\zz_p$ denotes the $p$-adic integers, and
\[
N(V) = \frac{|\{\text{symmetric, bilinear, perfect } \varphi\colon V\times V\to\cc^*\}|}{|V|\cdot|\Aut(V)|}.
\]
\end{theorem}

Here, a symmetric, $\zz$-bilinear map $\varphi\colon V\times V\to\cc^*$ is \emph{perfect} if the only $v\in V$ with $\varphi(v,V)=1$ or $\varphi(V,v)=1$ is $v=0$. For a discussion of the significance of the normalising factor $|V|\cdot|\Aut(V)|$, see \cite[\S1]{Wood:2017tk}, where it is also noted that, if $V$ is a finite abelian $p$-group
\[
V=\bigoplus_{i=1}^M \zz/p^{\lambda_i}\zz
\]
with $\lambda_1\ge\lambda_2\ge\dots\ge\lambda_M$, then
\[
N(V) = p^{-\sum_i\frac{\mu_i(\mu_i+1)}{2}}\prod_{i=1}^{\lambda_1}\prod_{j=1}^{\lfloor\frac{\mu_i-\mu_{i+1}}{2}\rfloor}(1-p^{-2j})^{-1},
\]
where $\mu_i=|\{j\mid \lambda_j\ge i\}|$. Consequently, for large $p$,
\[
\lim_{n\in2\nn}\pp\left(K_0\left(\cs(\hat\g_{n,r})\right)\otimes \zz_p\cong \zz/p^N\zz\right) \approx p^{-N},
\]
while
\[
\lim_{n\in2\nn}\pp\left(K_0\left(\cs(\hat\g_{n,r})\right)\otimes \zz_p\cong (\zz/p\zz)^N\right) \approx p^{-\frac{N(N+1)}{2}}.
\]
As for the left hand side of (\ref{eqn:kd}) when $P=\{p\}$, note that $K_0\left(\cs(\hat\g_{n,r})\right)\otimes\zz_p\cong V$ precisely when $K_0\left(\cs(\hat\g_{n,r})\right)$ is a finite abelian group whose $p$-Sylow subgroup is isomorphic to $V$. This means that, compared to its counterpart $(\zz/p\zz)^N$, the cyclic group $\zz/p^N\zz$ is much more likely to appear in the $K_0$-group of a random graph algebra $\cs(\hat\g_{n,r})$.

\begin{proof}[Proof of Theorem~\ref{thm:ck}]
By construction, $\hat\g_{n,r}$ has neither sinks nor sources, and every vertex is the starting point of at least $r\ge3$ loops, so by Proposition~\ref{prop:pisun}(\ref{item:pi1}), $\cs(\hat\g_{n,r})$ is isomorphic to a purely infinite Cuntz-Krieger algebra. Since $\hat\g_{n,r}$ is a.a.s.\ connected (see \cite{Lavrov:2022se}), it follows from Proposition~\ref{prop:pisun}(\ref{item:pi2}) that $\cs(\hat\g_{n,r})$ is a.a.s.\ simple.

The second statement is proved in exactly the same way as \cite[Theorem 1.5]{Nguyen:2018vh}, but with the (symmetric) adjacency matrix $C_n$ replaced by $C_n-I$. The key point is that, by Wood's incredibly powerful limiting distribution/moment machinery \cite[Theorem 6.1, Theorem 8.3 and Corollary 9.2]{Wood:2017tk}, it suffices to show that the expected number of surjections from $\coker(C_n-I)$ to $V$ tends to the size $|\wedge^2V|$ of the exterior power of $V$ as $n\in2\nn$ tends to $\infty$. Here is a brief summary of the argument.

Write $I(V) = \Span_\zz\{v\otimes v \mid v\in V\}$, so that $\wedge^2V = (V\otimes V)/I(V)$. For $q=(q_1,\dots,q_n)\in V^n=\Homo(\zz^n,V)$, let $\mc$ be the minimal coset of $V$ containing $\{q_1,\dots,q_n\}$, that is,
\[
\mc = q_n+\Span_\zz\{q_i-q_n \mid 1\le i\le n-1\}.
\]
For $k\in\nn$, define
\[
R^S(q,k) = \left\{s\in(k\cdot\mc)^n \mid \sum_{i=1}^n q_i\otimes s_i\in I(V) \text{ and } \sum_{i=1}^n s_i = k\sum_{i=1}^n q_i \right\}.
\]
For every $q\in V^n$, we have
\[
q \in R^S(q,r) \iff 0 \in R^S(q,r-1)
\]
and, since $C_n$ is obtained from $r$ perfect matchings on $\{1,\dots,n\}$, it is not hard to see that $(C_n-I)q\in R^S(q,r-1)$. By \cite[Theorem 1.6]{Meszaros:2020wi} and the computations that appear in the proof of \cite[Theorem 1.5]{Nguyen:2018vh},
\begin{align*}
\lim_{n\in2\nn}\ee(|\sur(\coker(C_n-I),V)|) &= \lim_{n\in2\nn}\sum_{\substack{q\in\sur(\zz^n,V)\\q\in R^S(q,r)}} \pp(C_nq=q)\\
&= \lim_{n\in2\nn}\sum_{\substack{q\in\sur(\zz^n,V)\\q\in R^S(q,r)}} |R^S(q,r)|^{-1}\\
&= \lim_{n\in2\nn}\sum_{\substack{q\in\sur(\zz^n,V)\\0\in R^S(q,r-1)}} \frac{|\wedge^2V|}{|V|^{n-1}}\\
&= \lim_{n\in2\nn}(|V|^{n-1} + o(|V|^n))\frac{|\wedge^2V|}{|V|^{n-1}}\\
&= |\wedge^2V|. \qedhere
\end{align*}
\end{proof}

\end{document}